\documentclass[11pt,a4paper]{article}
\usepackage[utf8]{inputenc}
\usepackage[T1]{fontenc}
\usepackage{lmodern}
\usepackage{enumerate}
\usepackage[english]{babel}
\usepackage{amsmath}
\usepackage{amssymb}
\usepackage{amsthm}
\usepackage{nicefrac}
\usepackage{todonotes}
\usepackage{gensymb}
\usepackage{float}
\graphicspath{{figures/}}
\usepackage{subcaption}
\usepackage{placeins}
\usepackage{fullpage}

\usepackage[colorlinks=true]{hyperref}
\hypersetup{urlcolor=blue, citecolor=blue, linkcolor=blue}

\newtheorem{thm}{Theorem}[section]
\newtheorem{lem}[thm]{Lemma}
\newtheorem{prop}[thm]{Proposition}
\newtheorem{cor}[thm]{Corollary}
\newtheorem{rem}[thm]{Remark}

\newtheorem{ass}{Assumption}

\theoremstyle{definition}
\newtheorem{definition}[thm]{Definition}

\numberwithin{equation}{section}
\numberwithin{figure}{section}

\newcommand{\norm}[1]{\left \|  #1 \right \|}
\newcommand{\abs}[1]{\left|  #1 \right|}
\DeclareMathOperator*{\argmin}{argmin}

\newcommand{\HS}{\mathop{\mathrm{HS}}}
\newcommand{\Cov}{\mathop{\mathrm{Cov}}}

\newcommand{\bs}{\boldsymbol}

\newcommand{\bfm}{{\bs m}}
\newcommand{\bfn}{{\bs n}}
\newcommand{\bfp}{{\bs p}}

\newcommand{\bfu}{{\bs u}}

\newcommand{\bfx}{{\bs x}}
\newcommand{\bfy}{{\bs y}} 
\newcommand{\bfz}{{\bs z}}

\newcommand{\bfT}{{\bs T}}

\newcommand{\field}[1]{{\mathbb{#1}}}
\newcommand{\C}{\field{C}}
\newcommand{\N}{\field{N}}
\newcommand{\R}{\field{R}}

\newcommand{\E}{\mathbb{E}}

\newcommand{\M}{\mathbb{M}}

\newcommand{\Acal}{\mathcal{A}}

\newcommand{\Ccal}{\mathcal{C}}

\newcommand{\Hcal}{\mathcal{H}}
\newcommand{\Ical}{\mathcal{I}}

\newcommand{\Sd}{{S^{d-1}}}
\newcommand{\Rd}{{\R^d}}

\newcommand{\Rdd}{\R^{d\times d}}

\newcommand{\ds}{\, \dif s} 

\newcommand{\dx}{\, \dif \bfx}
\newcommand{\dy}{\, \dif \bfy} 
\newcommand{\dz}{\, \dif \bfz}

\newcommand{\rmi}{\mathrm{i}} 

\newcommand{\eps}{\varepsilon}
\newcommand{\di}{\partial}

\newcommand{\ph}{\,\cdot\,}

\newcommand{\normm}{\abs{\bfm}}
\newcommand{\tm}{\subset}
\newcommand{\Llra}{\Longleftrightarrow}

\newcommand{\ol}{\overline}
\newcommand{\ul}{\underline}

\newcommand{\loc}{{\mathrm{loc}}}

\newcommand{\gvec}{\ul{g}}
\newcommand{\pvec}{\ul{p}}
\newcommand{\wvec}{\ul{w}}

\newcommand{\Cmat}{\ul{C}}

\newcommand{\npos}{M_0}

\newcommand{\Ifac}{\ul{\Ical}^{\mathrm{fac}}}
\newcommand{\Icap}{\ul{\Ical}^{\mathrm{Cap}}}
\newcommand{\Icbf}{\ul{\Ical}^{\mathrm{cbf}}}
\newcommand{\Icbfdr}{\ul{\Ical}^{\mathrm{cbf+dr}}}

\newcommand{\vhat}{\widehat{v}}

\DeclareMathOperator{\diag}{diag}
\DeclareMathOperator{\supp}{supp}
\DeclareMathOperator{\dif}{d\!}
\DeclareMathOperator{\ran}{ran}

\DeclareMathOperator*{\essinf}{ess\,inf}

\DeclareMathOperator{\var}{Var}

\DeclareMathOperator{\inn}{inn}
\DeclareMathOperator{\out}{out}

\author{
  Roland Griesmaier\footnote{Institut f\"ur
    Angewandte und Numerische Mathematik, 
    Karlsruher Institut f\"ur Technologie, Englerstr.~2,
    76131 Karlsruhe, Germany (\nolinkurl{roland.griesmaier@kit.edu})}
  \ and
  Hans-Georg Raumer\thanks{Institut für Aerodynamik und
    Strömungstechnik, Deutsches Zentrum für Luft- und Raumfahrt (DLR),
    Bunsenstraße 10, 37073 Göttingen, Germany, 
    (\nolinkurl{hans-georg.raumer@dlr.de})} \footnote{Corresponding author.}
}

\title{The factorization method and Capon’s method for random source
  identification in experimental aeroacoustics} 

\begin{document}

\maketitle

\begin{abstract}
  Experimental aeroacoustics is concerned with the estimation of
  acoustic source power distributions, which are for instance caused
  by fluid structure interactions on scaled aircraft models inside a
  wind tunnel, from microphone array measurements of associated
  sound pressure fluctuations. 
  In the frequency domain aeroacoustic sound propagation can be
  modelled as a random source problem for a convected Helmholtz
  equation. 
  This article is concerned with the inverse random source problem to
  recover the support of an uncorrelated aeroacoustic source from
  correlations of observed pressure signals. 
  We show that a variant of the factorization method from inverse
  scattering theory can be used for this purpose. 
  We also discuss a surprising relation between the factorization
  method and a commonly used beamforming algorithm from experimental
  aeroacoustics, which is known as Capon's method or as the minimum
  variance method.  
  Numerical examples illustrate our theoretical findings. 
\end{abstract}

{\small\noindent
  Mathematics subject classifications (MSC2010):
  35R30,  
  (65N21)  
  \\\noindent
  Short title: Random source identification in aeroacoustics
}

\section{Introduction}
\label{sec:Introduction}
In experimental aeroacoustic testing a solid object (e.g., a model of
an aircraft component) is placed inside a wind tunnel, and the fluid
structure interaction between the flow field and the object generates
sound pressure fluctuations, i.e., aeroacoustic noise.
The raw acoustic time signal is recorded by an array of microphones and
further post-processed to obtain correlation data in the frequency
domain. 
Based on these correlation data one then seeks to localize and
quantify the power distribution of the aeroacoustic sources (see,
e.g., \cite{BilKin76,Merino2019,Underbrink2002}). 

In this work we restrict the discussion to subsonic homogeneous
unidirectional flow fields in free space, and we use the convected
Helmholtz equation to model the propagation of time-harmonic
aeroacoustic pressure waves. 
These waves and the associated sources are usually considered as
random functions. 
Following \cite{Hohage2020} we model the aeroacoustic acoustic
pressure signal as a Hilbert space process with zero mean and a
covariance operator that acts as a multiplication operator. 
The inverse source problem then amounts to reconstructing the source
power function from the corresponding covariance operator of the
aeroacoustic pressure signal on the microphone array.
In practise the latter can be estimated from microphone array
measurements by Welch's method \cite{Welch1967}.
In our analysis we assume that the covariance operator corresponding
to an idealized continuum model for the microphone array is available.  
It has been shown in \cite{Hohage2020} that this inverse random source 
problem has a unique solution. 
For further results on inverse random source problems for
time-harmonic acoustic waves, which are not directly related to
aeroacoustic imaging, we refer, e.g., to
\cite{BaoChePei16,BaoChoLiZho14,Dev79,LiHelLi20}. 

Various reconstruction procedures have been discussed for correlation
based random source identification in aeroacoustics. 
Covariance fitting (see, e.g., \cite{Blacodon2004,Yardibi2008})
estimates source powers directly from correlation data of the
observed acoustic random pressure signal by minimizing a suitably
regularized output least squares functional. 
A faster and therefore more popular reconstruction technique in
experimental aeroacoustics is beamforming (see, e.g.,
\cite{Borcea11,Cox1987,Raumer2021,Shan1985,Van88}). 
Instead of solving the inverse source problem for all source positions
at once, beamforming estimates the source powers at individual source
positions separately. 
In particular DAMAS~\cite{Brooks2006} and CLEAN-SC~\cite{Sijtsma2007},
which combine beamforming methods with suitable postprocessing schemes
to improve the spatial resolution of the reconstruction, have become
standard tools. 
Both, covariance fitting and beamforming, have recently been
reviewed from a continuous perspective in \cite{Hohage2020}.

In this work we focus on the localization of extended aeroacoustic
source power functions. 
We show that a variant of the factorization method from inverse
scattering theory can be used to recover the support of a random
source from correlations of aeroacoustic pressure fluctuations.  
The factorization method has been introduced in the framework of
inverse obstacle scattering~\cite{Kirsch1998} and inverse
medium scattering~\cite{Kirsch1999} by Kirsch.
It has subsequently been attracting a considerable amount of attention
over the past twenty-five years. 
We will show that the mathematical structure of the covariance
operator of the aeroacoustic pressure signal is closely related to the
structure of the Born approximation of the far field operator for the
inverse medium scattering problem (see, e.g., \cite{Kirsch2017}). 
This will be used to establish the theoretical foundation of the
factorization method for the aeroacoustic inverse source problem.
On the other hand, we will see that the inf-criterion of the
factorization method and also the traditional imaging functional that
is obtained from Picard's criterion (see, e.g,~\cite{Kirsch2008}) is
closely related to another well-established beamforming algorithm that
is known as Capon's method or as the minimum variance method (see,
e.g., \cite{Capon1969,Li2003,Lorenz2005}). 
In particular, our results give a mathematically rigorous theoretical
interpretation of the reconstructions obtained by Capons's method.
We show for the first time that (for our idealized measurement setup
and in the absence of measurement errors) Capon's method recovers the
correct support of locally strictly positive source power functions. 

This article is organized as follows.
In Section~\ref{sec:aeroacoustic_ip} we briefly recall some basic
facts on solutions to the convected Helmholtz equation, and we
introduce the stochastic model for the aeroacoustic source problem
with uncorrelated extended sources. 
In Section~\ref{sec:factorization_aeroacoustics} we establish the main
result of this work, which is a theoretical justification of the
factorization method for reconstructing the support of the source
power function from the covariance operator corresponding to the
radiated sound pressure fluctuations.
In Section~\ref{sec:CaponsMethod} we discuss the relation between
Capon's method and the factorization method.
Some numerical results on experimental data are presented in
Section~\ref{sec:numerical_examples}.

\section{The aeroacoustic inverse source problem}
\label{sec:aeroacoustic_ip}
Let $\M\subset \Sigma_0 := \{ \bfx\in\Rd \;:\; x_d = 0\}$ and 
$\Omega \subset \R^d_+ := \{ \bfx\in\Rd \;:\; x_d > 0\}$, $d=2,3$, be
relatively open domains such that 
$\ol{\Omega}\cap \Sigma_0 = \emptyset$.
In the following $\M$ represents an idealized~$(d-1)$-dimensional
continuous measurement array, and $\Omega$ is supposed to be a region
in space that contains all possible aeroacoustic sources.

\subsection{The convected Helmholtz equation}
The basic sound propagation model that is used in experimental 
aeroacoustics to describe time-harmonic sound waves inside a 
subsonic homogeneous flow field $\bfu \in \Rd$ is the convected
Helmholtz equation.
Given a \emph{source term} $Q \in L^2(\Omega)$, the associated
\emph{sound pressure field} $p$ satisfies 
\begin{equation}
  \label{eq:ConvectedHelmholtz}
  \Delta p + (k + \rmi \bfm \cdot \nabla)^2p 
  \,=\,  -Q  \qquad \text{in } \Rd \,,
\end{equation}
where $k := \omega / c$ is the \emph{wave number}, $
\omega$ the \emph{frequency}, and $c$ the \emph{speed of sound}.
Here, subsonic means that the \emph{Mach vector} $\bfm := \bfu / c$
satisfies $\normm < 1$.
In the following we also use the
notation~$\beta := \sqrt{1 - \normm^2}$. 
Throughout, $|\ph|$ denotes the Euclidean norm on $\Rd$. 

We will assume that the \emph{convective field} $\bfu$ is aligned with
the $x_1$-direction, i.e., that 
\begin{equation*}
  \bfm
  \,=\, \left(m_1, 0, \dots, 0 \right)^{\top}
  \qquad \text{for some } m_1 = \normm \in [0,1) \,.
\end{equation*}
Solutions of the convected Helmholtz equation are linked to
solutions of the standard Helmholtz equation (i.e.,
\eqref{eq:ConvectedHelmholtz} with $\bfm\equiv0$) by the Lorentz
transformation.

\begin{prop}
  \label{prop:lorentztrafo}
  Let $\bfT := \diag\left(1/\beta, 1, \ldots, 1 \right)\in \Rdd$, 
  suppose that $U\subset \Rd$ is open, and let~$Q\in L^2(U)$. 
  Then $w_{\bfm} \in H^1(U)$ is a weak solution to the convected
  Helmholtz equation
  \begin{equation*}
    \Delta w_{\bfm} + (k + \rmi \bfm \cdot \nabla)^2 w_{\bfm} 
    \,=\,  -Q  \qquad \text{in } U
  \end{equation*}
  if and only if 
  \begin{equation*}
    w_{0}(\bfx)
    \,:=\, \exp\left( \frac{\normm \rmi k}{\beta} x_1 \right)
    w_{\bfm}\left( \bfT^{-1} \bfx  \right) \,,
    \qquad \bfx\in \bfT(U) \,, 
  \end{equation*}
  satisfies 
  \begin{equation*}
    \Delta w_0(\bfx) +  \frac{k^2}{\beta^2}w_0(\bfx)
    \,=\, -\exp\left( \frac{\normm\rmi k}{\beta} x_1 \right)
    Q(\bfT^{-1}\bfx) \,,
    \qquad \bfx\in\bfT(U)  \,,
  \end{equation*}
  i.e., $w_0$ is a weak solution to a standard Helmholtz equation with
  wavenumber $k/\beta$. 
\end{prop}

\begin{proof}
  This may be verified by direct calculation. 
\end{proof}

Using Proposition~\ref{prop:lorentztrafo}, the Sommerfeld radiation
condition, which determines outgoing solutions to the standard
Helmholtz equation on unbounded domains (see, e.g.,
\cite[p.~18]{Colton2019}), can be transferred to the convective
Helmholtz equation. 
Let $U \subset \Rd$ be a bounded domain, and let~$p \in
C^2(\Rd\setminus \ol{U})$ be a solution to 
\begin{equation*}
  \Delta p  + (k + \rmi \bfm \cdot \nabla)^2p 
  \,=\,  0 \qquad \text{in } \Rd\setminus\ol{U} \,.
\end{equation*}
Then we call $p$ \emph{radiating} if it satisfies the radiation
condition 
\begin{equation*}
  \lim_{r\to\infty} r^{\frac{d-1}{2}}
  \left( \left(\frac{\di}{\di r}
      - \rmi \frac{k}{\beta}\right)
    \exp\left( \frac{\normm \rmi k}{\beta} x_1 \right)
    p \left( \bfT^{-1} \bfx  \right) \right)
  \,=\, 0 \,, \qquad r=\abs{\bfx} \,, 
\end{equation*}
uniformly with respect to all directions $\bfx/\abs{\bfx} \in \Sd$.

Similarly, using Proposition~\ref{prop:lorentztrafo} the fundamental
solution of the convected Helmholtz equation can be obtained from the
fundamental solution for the standard Helmholtz equation (see, e.g.,
\cite[p.~19 and p.~89]{Colton2019}). 
To simplify the notation, we define the \emph{Mach norm} on $\Rd$ by 
\begin{equation*}
  \abs{\bfx}_{\bfm}
  \,:=\, \sqrt{(\bfx \cdot \bfm)^2 + \beta^2 \abs{\bfx}^2} \,,
  \qquad \bfx\in\Rd \,.
\end{equation*}
Therewith, the fundamental solution of the convected Helmholtz
equation is given by
\begin{equation}
  \label{eq:greensfunc}
  g(\bfx,\bfy)
  \,:=\, \exp\left(-\frac{\rmi k}{\beta^2} (\bfx-\bfy)\cdot\bfm\right)
  \cdot 
  \begin{cases}
    \frac{i}{4 \beta} H^{(1)}_0
    \left( \frac{k}{\beta^2} \abs{\bfx-\bfy}_{\bfm} \right)
    & \text{if } d=2 \,,\\[0.5em]
    \frac{1}{4\pi \abs{\bfx-\bfy}_{\bfm}}
    \exp\left(\frac{\rmi k}{\beta^2}\abs{\bfx-\bfy}_{\bfm}\right)
    & \text{if } d=3 \,,
  \end{cases}  
\end{equation}
for $\bfx,\bfy\in\Rd$, $\bfx\not=\bfy$. 
As usual, $H^{(1)}_0$ denotes the Hankel function of the first kind of 
order zero.
For later reference, we note that
\begin{equation}
  \label{eq:green_upperbnd}
  \abs{g(\bfx, \bfy)}
  \,\leq\, C(d) \abs{\bfx- \bfy}^{\frac{1-d}{2}}
  \qquad \text{for } \bfx,\bfy\in\Rd \,,\; \bfx \neq \bfy \,,
\end{equation}
with a constant $C(d)$ that depends only on the spatial dimension $d$.
Using the norm equivalence of $\abs{\ph}$ and $\abs{\ph}_{\bfm}$ on
$\Rd$, this bound follows directly from \eqref{eq:greensfunc} when
$d=3$, while for $d=2$ one uses the asymptotic behavior of
the Bessel functions (see, e.g., \cite[pp.~89--90]{Colton2019}). 

\begin{lem}
  \label{lem:WellPosedness}
  Let $Q\in L^2(\Omega)$.
  Then, the unique radiating solution $p\in H^1_\loc(\Rd)$ of
  \eqref{eq:ConvectedHelmholtz} is given by
  \begin{equation*}
    p(\bfx)
    \,=\, \int_\Omega Q(\bfy) g(\bfx,\bfy) \dy \,,
    \qquad \bfx\in\Rd \,.
  \end{equation*}
  Furthermore, $\bfp$ is real analytic in $\Rd\setminus\ol{\Omega}$. 
\end{lem}

\begin{proof}
  This follows from the one-to-one correspondence between radiating
  solutions to the standard Helmholtz equation and radiating
  solutions to the convected Helmholtz equation by means of the
  Lorentz transformation.
  The existence and uniqueness of radiating solutions to the
  corresponding source problem for the standard Helmholtz equation
  follows from Rellich's lemma (see, e.g.,
  \cite[Lmm~2.12]{Colton2019}) and the properties of the volume
  potential (see, e.g., \cite[Thms.~8.1--8.2]{Colton2019}).
  The real analyticity of $\bfp$ in $\Rd\setminus\ol{\Omega}$ follows
  from \cite[Prop.~3.4]{Hohage2020}. 
\end{proof}

The next proposition gives an integration by parts formula that is a
consequence of Green's second theorem (see, e.g.,
\cite[p.~19]{Colton2019}).
A complete proof can be found in Appendix~A of \cite{RaumerPhD}. 
\begin{prop}
  \label{prop:intbyparts}
  Let $U\subset \Rd$ be a bounded domain of class $C^1$ and let
  $\bfn = (n_1, \dots, n_d)^{\top}$ denote the unit outward normal
  vector on the boundary $\di U$.
  Then, for $p,w \in C^2(\ol{U})$ we have
  \begin{equation}
    \label{eq:intbyparts}
    \begin{split}
      &\int_U p(\bfy) \left(
        \Delta w (\bfy) + (k + \rmi \bfm \cdot \nabla)^2w(\bfy)
      \right) \dy
      - \int_U \left(
        \Delta p (\bfy) + (k - \rmi \bfm \cdot \nabla)^2p(\bfy)
      \right) w(\bfy) \dy \\
      &\phantom{\,=\,}
      + \oint_{\di U}  \left( p(\bfy) \frac{\di w}{\di\bfn} (\bfy)
        - w(\bfy) \frac{\di p}{\di\bfn}(\bfy) \right) \ds(\bfy)
      + 2\rmi k \normm \oint_{\di U} p(\bfy) w(\bfy) n_1(\bfy) \ds(\bfy)\\
      &\phantom{\,=\,}
      + \normm^2 \oint_{\di U} \left(
        w(\bfy) \frac{\di p}{\di y_1}(\bfy) n_1(\bfy)
        - p(\bfy) \frac{\di w}{\di y_1}(\bfy) n_1(\bfy) \right) \ds(\bfy) \,. 
    \end{split}
  \end{equation}
\end{prop}

Finally, we transfer the Helmholtz representation formula for
radiating solutions of the standard Helmholtz equation (see, e.g.,
\cite[Thm.~2.5]{Colton2019}) to radiating solutions of the convected
Helmholtz equation.
Again a complete proof, which employs
Proposition~\ref{prop:lorentztrafo}, can be found in Appendix~A of
\cite{RaumerPhD}. 

\begin{prop}
  \label{prop:representation_formula}
  Suppose that $U\subset \Rd$ is the open complement of an unbounded
  domain of class~$C^2$ and let $\bfn = (n_1, \dots, n_d)^{\top}$
  denote the unit outward normal vector on the boundary $\di U$. 
  Let $p \in C^2(\Rd\setminus\ol{U})\cap C^1(\Rd\setminus U)$ be
  a radiating solution to
  \begin{equation*}
    \Delta p  + (k + \rmi \bfm \cdot \nabla)^2p 
    \,=\,  0 \qquad \text{in } \Rd\setminus\ol{U} \,.
  \end{equation*}
  Then, for any $\bfx \in \Rd \setminus \ol{U}$, we have
  \begin{equation}
    \label{eq:representation_formula}
    \begin{split}
      p(\bfx) 
      &\,=\, \oint_{\di U} \left(
        p({\bfy}) \frac{\di g(\bfx,\bfy)}{\di\bfn(\bfy)}
        - g(\bfx,\bfy) \frac{\di p}{\di\bfn}(\bfy) \right)
      \cdot \bfn \ds({\bfy}) \\
      &\phantom{\,=\,}
      + \normm ^2 \oint_{\di U} \left(
        g(\bfx,\bfy) \frac{\di p}{\di y_1}(\bfy)
        - p(\bfy) \frac{\di g(\bfx,\bfy)}{\di y_1} \right)
      n_1(\bfy) \ds(\bfy) \\
      &\phantom{\,=\,}
      - 2 \normm \rmi k \oint_{\di U} p(\bfy)g(\bfx,\bfy) n_1(\bfy)
      \ds(\bfy) \,. 
    \end{split}
  \end{equation}
\end{prop}

\subsection{The random source process}
In experimental aeroacoustics sources are usually considered as random
functions. 
Following~\cite{Hohage2020}, we use a Hilbert space process, i.e., a
bounded linear operator
\begin{equation*}
  Q: L^2(\Omega) \to L^2(X, \Acal, \mathbb{P}) \,,
\end{equation*}
where $(X, \Acal, \mathbb{P})$ is the underlying probability space, to
model the source problem. 
Then, the associated \emph{random pressure signal} is given by 
\begin{equation}
  \label{eq:randomsignal}
  p(\bfx)
  \,=\, Q(g(\bfx,\ph)) \,, \qquad \bfx\in\Rd \,,
\end{equation}
where $g$ is the fundamental solution from \eqref{eq:greensfunc}.
Using \eqref{eq:green_upperbnd} we see that $g(\bfx,\ph)$ is square
integrable for any $\bfx\in\M$, and thus \eqref{eq:randomsignal} is
well-defined. 

The \emph{expectation} of $Q$ is the unique element
$\E[Q] \in L^2(\Omega)$ such that 
\begin{equation*}
  \left \langle \E[Q], v \right \rangle_{L^2(\Omega)}
  \,=\, \E\left( Qv  \right) 
  \qquad \text{for all } v \in L^2(\Omega) \,,
\end{equation*}
and the \emph{covariance operator}
$\Cov[Q]: L^2(\Omega) \to L^2(\Omega)$ is the unique
self-adjoint and positive-semidefinite operator that satisfies 
\begin{equation*}
  \left\langle \Cov[Q]\phi_1, \phi_2 \right\rangle_{L^2(\Omega)}
  \,=\, \Cov \left( Q\phi_1, Q\phi_2  \right)
  \qquad \text{for all } \phi_1, \phi_2 \in L^2(\Omega) \,.
\end{equation*}

It is commonly assumed in experimental aeroacoustics that the
random source has zero mean and is spatially uncorrelated.
\begin{ass}
  \label{as:source_process}
  The Hilbert space process $Q$ satisfies 
  \begin{enumerate}[(a)]
  \item $\E[Q] = 0$,
  \item and there is a $q \in L^\infty(\Omega)$, the
    \emph{source power function}, such that $\Cov[Q] = M_q$, where
    $M_q: L^2(\Omega) \to L^2(\Omega)$ denotes the multiplication
    operator given by 
    \begin{equation*}
      (M_qv)(\bfx)
      \,:=\, q(\bfx) v(\bfx) \,,
      \qquad \bfx\in\Omega \,.
    \end{equation*}
  \end{enumerate}
\end{ass}
We note that in the special case when $q\equiv 1$, a process $Q$ that
satisfies Assumption~\ref{as:source_process} is called a white noise
process. 
Since~$\Cov[Q]$ is symmetric and positive-semidefinite, the source
power function~$q$ is real-valued and nonnegative a.e.\ in $\Omega$. 
For any $\bfx\in\M$, the pressure signal~$p(\bfx)$ is a scalar,
complex random variable with $\E[p(\bfx)]=0$, and the correlation
between two observation positions $\bfx, \bfy \in \M$ satisfies
\begin{equation*}
  \begin{split}
    \Cov(p(\bfx), p(\bfy))
    &\,=\, \Cov \left( Q(g(\bfx, \cdot)), Q(g(\bfy, \cdot))\right)
    \,=\, \left \langle \Cov[Q]g(\bfx, \cdot),
      g(\bfy, \cdot) \right \rangle_{L^2(\Omega)} \\
    &\,=\, \left \langle M_q g(\bfx, \cdot),
      g(\bfy, \cdot) \right \rangle_{L^2(\Omega)}
    \,=\, \int_{\Omega} q(\bfz) g(\bfx, \bfz) \ol{g(\bfy, \bfz)} \dz \\
    &\,=:\, c_q(\bfx, \bfy) \,.
  \end{split}
\end{equation*}
Accordingly, the \emph{covariance operator} of the aeroacoustic
pressure signal $\Ccal(q): L^2(\M) \to L^2(\M)$ is given by 
\begin{equation*}
  \left( \Ccal(q) \psi \right)(\bfx)
  \,:=\, \int_{\M} \psi(\bfy) c_q(\bfx, \bfy) \dy \,,
  \qquad \bfx\in\M \,.
\end{equation*}
Using \eqref{eq:green_upperbnd} it follows that, for any
$q\in L^\infty(\Omega)$, the covariance operator $\Ccal(q)$ is a
Hilbert-Schmidt operator (see \cite[Pro.~2.2]{Hohage2020}). 

In experimental aeroacoustics finite dimensional approximations of the 
covariance operator~$\Ccal(q)$ are obtained from microphone array
measurements by estimating the covariance matrix of the microphone
signals. 
This estimation is usually carried out by Welch’s
method~\cite{Welch1967}. 
We are interested in the inverse source problem to reconstruct the
support of the source power function $q\in L^\infty(\Omega)$ from
observations of $\Ccal(q)\in \HS(L^2(\M))$.
We note that in \cite{Hohage2020} it has been established that in fact
even $q$ is uniquely determined by $\Ccal(q)$. 
In this work we will show that a variant of the factorization method
from inverse scattering theory can be utilized to recover the support
of $q$ from $\Ccal(q)$.

\section{The factorization method in aeroacoustic source imaging}
\label{sec:factorization_aeroacoustics}
From now on we let $q\in L^\infty(\Omega)$ with $q\geq 0$ a.e.\ in
$\Omega$ be a fixed source power function, and we denote by $\Ccal(q)$
the associated covariance operator. 
Following \cite{HarUlr13,KusSyl03} we distinguish the support, the
inner support, and the outer support of~$q$. 
These notions will be used in the characterization of the support of
$q$ in terms of $\Ccal(q)$ below. 

\begin{definition}
  \label{dfn:supports}
  Let $q\in L^\infty(\Omega)$ with $q\geq 0$.
  We identify $q$ with its extension to $\Rd$ by zero, and we define
  \begin{enumerate}[(a)]
  \item the \emph{support} $\supp(q)$ of $q$ as the complement of the
    union of all open subsets $U\tm\Rd$ such that
    $q|_U\equiv 0$. 
  \item the \emph{inner support} $\inn\supp(q)$ of $q$ as the union
    of all open subsets $U\tm\Rd$ such that $\essinf q|_U > 0$.
  \item  the \emph{outer support} $\out\supp(q)$ of $q$ as the
    complement of the union of all open, connected and unbounded
    subsets $U\tm\Rd$ such that $q|_U\equiv 0$. 
  \end{enumerate}
\end{definition}

\begin{rem}
  \label{rem:LocStrPos}
  In the following we denote by $D\tm\Rd$ the interior of $\supp(q)$. 
  Following \cite{GebHyv07} we say that the source power function $q$
  is \emph{locally strictly positive} on $D$, if for each $\bfx\in D$
  there exist $\eps_\bfx,r_\bfx>0$ such that
  $B_{r_\bfx}(\bfx) \tm D$ and
  \begin{equation*}
    q(\bfx) > \eps_\bfx \qquad
    \text{for a.e.\ } \bfx \in B_{r_\bfx}(\bfx) \,,
  \end{equation*}
  where $B_{r_\bfx}(\bfx)$ denotes the ball of radius $r_\bfx$
  centered at $\bfx$.
  If this is the case, and if $\Rd\setminus\supp(q)$ is connected,
  then 
  \begin{equation*}
    \ol{\inn\supp(q)} \,=\, \supp(q) \,=\, \out\supp(q) 
  \end{equation*}
  (see \cite[Cor.~2.5]{HarUlr13}).
  In general, the outer support is basically the support plus the
  holes that cannot be connected to infinity. 
\end{rem}

Our goal is to reconstruct $D$ from the covariance operator 
$\Ccal(q)$ under minimal assumptions on $q$. 
The techniques that we use have been developed for time-harmonic
inverse scattering problems in
\cite{Kirsch1998,Kirsch1999,Kirsch2017,Kirsch2008}, and we further
apply ideas that have been proposed in
\cite{BruHanPid01,GebHan05,GebHyv07}. 
We define the operator $\Hcal_D: L^2(D) \to L^2(\M)$ by
\begin{equation*}
  (\Hcal_D\psi)(\bfx)
  \,:=\, \int_D \sqrt{q(\bfy)} g(\bfx,\bfy) \psi(\bfy) \dy \,,
  \qquad \bfx \in \M \,.
\end{equation*}
Then the adjoint $\Hcal_D^*: L^2(\M) \to L^2(D)$ of $\Hcal_D$ is given
by 
\begin{equation*}
  (\Hcal_D^*\phi)(\bfy)
  \,=\, \sqrt{q(\bfy)} \int_\M \ol{g(\bfx,\bfy)} \phi(\bfx) \dx \,,
  \qquad \bfy\in D \,.
\end{equation*}
Therewith, the covariance operator $\Ccal(q)$ can be decomposed
as 
\begin{equation*}
  \Ccal(q)
  \,=\, \Hcal_D \Hcal_D^* \,.
\end{equation*}

The following range identities are the first ingredient of our
reconstruction method. 
\begin{thm}
  \label{thm:range_identity}
  Suppose that $q\in L^\infty(\Omega)$, $q\geq0$ a.e.\ on $\Omega$. 
  \begin{enumerate}[(a)]
  \item The covariance operator $\Ccal(q)$ has a self-adjoint and
    positive-semidefinite square root $\Ccal(q)^{1/2}$, which satisfies 
    \begin{equation}
      \label{eq:range_identity}
      \ran\left(\Ccal(q)^{1/2}\right)
      \,=\,  \ran\left(\Hcal_D\right) \,. 
    \end{equation}
  \item  For any $\phi \in L^2(\M)$, $\phi\not=0$, 
    \begin{equation*}
      \ran\left(\Ccal(q)^{1/2}\right)
      \quad\Llra\quad
      \inf \left\{ \left\langle \psi, \Ccal(q) \psi \right\rangle_{L^2(\M)}
        \;:\; \psi\in L^2(\M)\,,\;
        \langle \psi, \phi \rangle_{L^2(\M)} = 1 \right\} > 0 \,. 
    \end{equation*}
  \end{enumerate}
\end{thm}

\begin{proof}
  \begin{enumerate}[(a)]
  \item Since $\Ccal(q)$ is self-adjoint and positive-semidefinite, its
    square root is well-defined (see, e.g., \cite[p.~44]{Engl1996}). 
    The range identity \eqref{eq:range_identity} has, e.g., been shown
    in \cite[Prop.~2.18]{Engl1996}. 
  \item Observing that $\Ccal(q) = \Ccal(q)^{1/2}\Ccal(q)^{1/2}$,
    this follows from \cite[Thm~1.16]{Kirsch2008}. 
  \end{enumerate}
\end{proof}

The second ingredient of our reconstruction method is the following
characterization of the support of the source power function in terms
of the point sources $g(\ph,\bfz)|_\M$, $\bfz\in\Omega$, and the
range of the operator $\Hcal_D$. 
\begin{thm}
  \label{thm:shape_characterization}
  Suppose that $q\in L^\infty(\Omega)$, $q\geq0$ a.e.\ on $\Omega$,
  and let $\bfz\in\Omega$. 
  \begin{enumerate}[(a)]
  \item If $\bfz \in \inn\supp(q)$,
    then $g(\ph,\bfz)|_\M \in \ran\left(\Hcal_D\right)$.
  \item If $\bfz \in \Omega\setminus\out\supp(q)$,
    then $g(\ph,\bfz)|_\M \not\in \ran\left(\Hcal_D\right)$.
  \end{enumerate}
\end{thm}

\begin{rem}
  If the source power function $q$ is locally strictly positive in the
  sense of Remark~\ref{rem:LocStrPos}, then
  Theorem~\ref{thm:shape_characterization} can be used to determine
  whether a \emph{sampling point} $\bfz\in\Omega$ belongs to $D$ or to
  $\Omega\setminus\ol{D}$.   
\end{rem}

\begin{proof}[Proof of Theorem~\ref{thm:shape_characterization}]
  \begin{enumerate}[(a)]
  \item Let $\bfz \in \inn\supp(q)$.
    Then there exists an $\eps>0$ such that $B_\eps(\bfz) \tm D$ and
    $\essinf(q|_{B_\eps(\bfz)})>0$.
    Choose $\eta \in C^{\infty}(\R)$ with $0\leq\eta \leq 1$, 
    $\eta(s)=0$ for $|s|\leq \eps/2$,
    and $\eta(s) = 1$ for $|s| \geq \eps$.
    We define $w\in C^\infty(\Rd)$ by
    \begin{equation*}
      w(\bfx)
      \,:=\, \eta(|\bfx-\bfz|) g(\bfx,\bfz) \,, \qquad
      \bfx\in\Rd \,,\; \bfx\not=\bfz \,.
    \end{equation*}
    Let $\phi \in L^2(D)$ be given by
    \begin{equation*}
      \phi
      \,:=\,
      \begin{cases}
        - \frac{1}{\sqrt{q}}
        \left( \Delta w + (k + \rmi \bfm \cdot \nabla)^2w \right)
        & \text{in } B_\eps(\bfz) \,,\\
        0
        & \text{in } D\setminus \ol{B_\eps(\bfz)} \,.
      \end{cases}
    \end{equation*}
    Then, using \eqref{eq:intbyparts}, we find for any $\bfx\in\M$ that
    \begin{equation}
      \label{eq:Proof_shape_characterization-1}
      \begin{split}
        \left( \Hcal_D \phi \right) (\bfx)
        &\,=\, - \int_{B_\eps(\bfz)} g(\bfx, \bfy) \left(
          \Delta w + (k + \rmi \bfm \cdot \nabla)^2w \right)(\bfy)
        \dy \\
        &\,=\, - \int_{B_\eps(\bfz)} \left(
          \Delta_y g(\bfx,\bfy)
          + (k - \rmi \bfm \cdot \nabla_y)^2g(\bfx,\bfy)
        \right) w(\bfy) \dy \\
        &\phantom{\,=\,}
        - \oint_{\di B_\eps(\bfz)}  \left( g(\bfx,\bfy)
          \frac{\di w}{\di\bfn} (\bfy) 
          - w(\bfy) \frac{\di g(\bfx,\bfy)}{\di\bfn(\bfy)} \right)
        \ds(\bfy)\\
        &\phantom{\,=\,}
        - 2\rmi k \normm \oint_{\di B_\eps(\bfz)} g(\bfx,\bfy)
        w(\bfy) n_1(\bfy) \ds(\bfy)\\ 
        &\phantom{\,=\,}
        - \normm^2 \oint_{\di B_\eps(\bfz)} \left(
          w(\bfy) \frac{\di g(\bfx,\bfy)}{\di y_1} n_1(\bfy)
          - g(\bfx,\bfy) \frac{\di w}{\di y_1}(\bfy) n_1(\bfy) \right)
        \ds(\bfy) \,. 
      \end{split}
    \end{equation}
    Since
    \begin{equation*}
      \Delta_y g(\bfx,\ph)
      + (k - \rmi \bfm \cdot \nabla_y)^2g(\bfx,\ph)
      \,=\, 0
      \qquad \text{in } B_\eps(\bfz) \,,    
    \end{equation*}
    the volume integral on the right hand side of
    \eqref{eq:Proof_shape_characterization-1} vanishes.
    Moreover, the function $w$ is a radiating solution of the
    homogeneous convected Helmholtz equation on
    $\Rd \setminus \ol{B_\eps(\bfz)}$.
    Hence,~\eqref{eq:representation_formula} can be applied to
    conclude that 
    \begin{equation*}
      \left( \Hcal_D \phi \right) (\bfx)
      \,=\, w(\bfx)
      \,=\, g(\bfx, \bfz) \,, \qquad \bfx\in\M \,.
    \end{equation*}
    This yields the assertion. 
    
  \item Suppose that $\bfz \in \Omega\setminus\out\supp(q)$, and that
    $g(\ph,\bfz)|_\M \in \ran(\Hcal_D)$.
    Then there is $\psi\in L^2(D)$ such that
    \begin{equation*}
      g(\ph,\bfz)|_\M
      \,=\, \Hcal_D\psi \qquad \text{in } \M \,.
    \end{equation*}
    Since $\M\tm\Sigma_0$ is relatively open and
    \begin{equation*}
      v(\bfx)
      \,:=\, g(\bfx,\bfz)
      - \int_D \sqrt{q(\bfy)} g(\bfx,\bfy) \psi(\bfy) \dy \,,
      \qquad \bfx \in \Rd\setminus\{\bfz\} \,,
    \end{equation*}
    is real analytic in $\Rd\setminus(\ol{D}\cup\{\bfz\})$, we find by
    analytic continuation 
    that $v|_{\Sigma_0} = 0$.
    Now we use the reflection principle and define
    \begin{equation*}
      \vhat(\bfx)
      \,:=\,
      \begin{cases}
        v(\bfx) \,, & \bfx\in\R^d_- \,,\\
        -v(x_1,\ldots,x_{d-1},-x_d) \,, & \bfx\in\R^d_+ \,,
      \end{cases}
    \end{equation*}
    where ${\R^d_- := \{ \bfx\in\Rd \;:\; x_d < 0\}}$.
    Recalling that $\Omega\tm\R^d_+$, we find that
    $\vhat\in H^1_\loc(\Rd)$ is an entire radiating solution to the
    convected Helmholtz equation
    \begin{equation*}
      \Delta \vhat + (k + \rmi \bfm \cdot \nabla)^2\vhat 
      \,=\,  0  \qquad \text{in } \Rd \,.
    \end{equation*}
    Thus, $\vhat$ must vanish identically on $\Rd$ (see
    \cite[p.~28]{Colton2019} for the corresponding result for the
    standard Helmholtz equation, and use one-to-one correspondence
    between radiating solutions to the standard Helmholtz equation and
    radiating solutions to the convected Helmholtz equation by means
    of the Lorentz transformation).
    Therefore, $v$ vanishes on~$\R^d_-$, and we find by
    analytic continuation that~$v$ is zero on
    $\Rd\setminus(\out\supp(q)\cup\{\bfz\})$.
    Here we used that $\Rd\setminus\out\supp(q)$ is connected. 
    This means that
    \begin{equation}
      \label{eq:Proof_shape_characterization-2}
      g(\bfx,\bfz)
      \,=\, \int_D \sqrt{q(\bfy)} g(\bfx,\bfy) \psi(\bfy) \dy \,,
      \qquad \bfx \in \Rd\setminus(\out\supp(q)\cup\{\bfz\}) \,.
    \end{equation}
    
    However, the left hand side of
    \eqref{eq:Proof_shape_characterization-2} is unbounded on 
    $B_{1}(\bfz) \cap (\Rd\setminus(\out\supp(q)\cup\{\bfz\}))$, while
    for the right hand side we obtain that
    \begin{equation*}
      \left| \int_D \sqrt{q(\bfy)} g(\bfx,\bfy) \psi(\bfy) \dy
      \right|^2
      \,\leq\, \norm{q}_{L^\infty(\Omega)}
      \norm{g(\bfx,\ph)}_{L^2(D)}^2 \norm{\psi}_{L^2(D)}^2
    \end{equation*}
    is uniformly bounded for 
    $\bfx \in B_{1}(\bfz) \cap
    (\Rd\setminus(\out\supp(q)\cup\{\bfz\}))$ by
    \eqref{eq:green_upperbnd}. This gives a contradiction, and thus 
    we have shown that
    $g(\ph,\bfz)|_\M \not\in \ran(\Hcal_D)$.
  \end{enumerate}
\end{proof}

Combining
Theorems~\ref{thm:range_identity}--\ref{thm:shape_characterization}
gives the following result.
\begin{cor}
  \label{cor:InfCriterion}
  Suppose that $q\in L^\infty(\Omega)$, $q\geq0$ a.e.\ on $\Omega$,
  and let $\bfz\in\Omega$. 
  \begin{enumerate}[(a)]
  \item If $\bfz \in \inn\supp(q)$, then
    \begin{equation}
      \label{eq:InfCriterion}
      \inf \left\{ \left\langle \psi, \Ccal(q) \psi \right\rangle_{L^2(\M)}
        \;:\; \psi\in L^2(\M)\,,\;
        \langle \psi, g(\ph,\bfz)|_\M \rangle_{L^2(\M)} = 1 \right\} > 0 \,.
    \end{equation}
  \item If $\bfz \in \Omega\setminus\out\supp(q)$, then the infimum in
    \eqref{eq:InfCriterion} is zero. 
  \end{enumerate}
\end{cor}

Since $\Ccal(q): L^2(\M) \to L^2(\M)$ is compact, self-adjoint, and
positive-semidefinite, it has a complete ortho\-normal eigensystem.
We assume that the (possibly finite) sequence of positive
eigenvalues~$(\lambda_j)_{j\in\N}$ is in decreasing order such that
each eigenvalue is repeated according to its multiplicity, and we
denote by $(\psi_j)_{j\in\N}$ the corresponding sequence of
ortho\-normal eigenfunctions. 
Accordingly, the nonzero eigenvalues and the corresponding
eigenvectors of $\Ccal(q)^{1/2}$ are given
by~$(\sqrt{\lambda_j})_{j\in\N}$ and~$(\psi_j)_{j\in\N}$, respectively. 

\begin{thm}
  \label{thm:MainResult}
  Suppose that $q\in L^\infty(\Omega)$, $q\geq0$ a.e.\ on $\Omega$,
  and assume that $\inn\supp(q)\not=\emptyset$.
  Let $\bfz\in\Omega$. 
  \begin{enumerate}[(a)]
  \item If $\bfz \in \inn\sup(q)$, then
    \begin{equation}
      \label{eq:PicardSeries}
      \sum_{j=1}^\infty \frac{\left|\left\langle g(\ph,\bfz),
            \psi_j \right\rangle_{L^2(\M)}\right|^2}{\lambda_j}
      \,<\, \infty \,.
    \end{equation}
  \item If $\bfz \in \Omega\setminus\out\supp(q)$,
    then the series in \eqref{eq:PicardSeries} does not converge. 
  \end{enumerate}
\end{thm}

\begin{proof}
  We first show that
  \begin{equation}
    \label{eq:Proof-MainResult-1}
    g(\ph,\bfy) \in \ol{\ran(\Ccal(q)^{1/2})} \qquad
    \text{for any } \bfy \in \Omega \,.
  \end{equation}
  To see this, let $\phi \in \ker(\Hcal_D^*)$, i.e.,
  \begin{equation*}
    0
    \,=\, (\Hcal_D^*\phi)(\bfy)
    \,=\, \sqrt{q(\bfy)} \int_\M \ol{g(\bfx, \bfy)} \phi(\bfx) \dx
    \qquad \text{for all } \bfy \in D \,.
  \end{equation*}
  By assumption there is an open subset $B\tm D$ such that
  $\essinf(q|_{B})>0$.
  Thus,
  \begin{equation*}
    \int_\M \ol{g(\bfx, \bfy)} \phi(\bfx) \dx = 0 
    \qquad \text{for any } \bfy \in B \,,
  \end{equation*}
  and by analytic continuation this holds even for any
  $\bfy \in \Omega$.
  Accordingly,
  \begin{equation*}
    g(\ph, \bfy) \in \ker(\Hcal_D^*)^\perp
    \,=\, \ol{\ran(\Hcal_D)}
    \qquad \text{for any } \bfy\in \Omega \,.
  \end{equation*}
  The range identity \eqref{eq:range_identity} gives
  \eqref{eq:Proof-MainResult-1}. 
  
  Therefore, combining \eqref{eq:range_identity} and
  Theorem~\ref{thm:shape_characterization}, and applying Picard's
  theorem (see, e.g., \cite[Thm.~4.8]{Colton2019}) yields the assertion
  of the theorem. 
\end{proof}

Usually in practice, only a finite number of microphones at positions
$\bfx_1, \dots , \bfx_M \in\M$ is available to measure the random
pressure fluctuations.
A self-adjoint, positive-semidefinite correlation matrix
$\Cmat \in \C^{M \times M}$, which approximates the covariance
operator $\Ccal(q)$, can be obtained from these observations using
Welch's method~\cite{Welch1967}.
We denote by $(\lambda_j, \ul{\psi}_j)_{1\leq j\leq M}$ an ortho\-normal
eigensystem of $\Cmat$ such that the eigenvalues are in decreasing
order and counted with multiplicity. 
Let $0< \npos \leq M$ be the number of positive eigenvalues.
Then we define the \emph{imaging functional}~$\Ifac:\, \Omega \to \R$
of the factorization method by 
\begin{equation}
  \label{eq:imagingfunc_fac_discrete1} 
  \Ifac(\bfz)
  \,:=\, \left( \sum_{j=1}^{\npos}
    \frac{\abs{\left\langle \gvec(\bfz), \ul{\psi}_j \right\rangle_{2}}^2}
    {\lambda_j} \right)^{-1} \,, \qquad \bfz\in\Omega \,,
\end{equation}
where
\begin{equation*}
  \gvec(\bfz)
  \,:=\, \left[ g(\bfx_1, \bfz), \ldots,
    g(\bfx_M, \bfz) \right]^\top \,, \qquad \bfz\in\Omega \,.
\end{equation*}
Denoting by $\Cmat^\dagger$ and $(\Cmat^{1/2})^\dagger$ the
pseudoinverses of $\Cmat$ and $\Cmat^{1/2}$, respectively,
\eqref{eq:imagingfunc_fac_discrete1} can be rewritten as
\begin{equation}
  \label{eq:imagingfunc_fac_discrete2} 
  \Ifac(\bfz)
  \,=\, \norm{(\Cmat^{1/2})^\dagger \gvec(\bfz) }_2^{-2}
  \,=\, \left( \gvec(\bfz)^* \Cmat^\dagger \gvec(\bfz) \right)^{-1} \,,
  \quad \bfz\in\Omega \,.
\end{equation}
According to Theorem~\ref{thm:MainResult}, the values of $\Ifac(\bfz)$
should be much smaller for $\bfz\in\Omega\setminus\out\supp(q)$ than
for $\bfz\in\inn\supp(q)$. 

The imaging functional in \eqref{eq:imagingfunc_fac_discrete2} is
closely related to Capon's method \cite{Capon1969} from seismic
imaging. 
In the context of correlation based aeroacoustic source mapping this
method is also known as the minimum variance method (see, e.g.,
\cite{Li2003,Lorenz2005}). 
This relationship is  quite surprising as Capon's method was
originally derived from a totally different viewpoint.
In the next subsection we discuss this observation in some more
detail.

\section{Capon's method}
\label{sec:CaponsMethod}
In aeroacoustic source identification imaging functionals
$\ul{\Ical}:\Omega\to\R$ are usually defined on a source
region~$\Omega \subset \Rd$ as introduced at the beginning of
Section~\ref{sec:aeroacoustic_ip}. 
Imaging procedures that map focus points $\bfz \in \Omega$ in the
source region directly to an image value $\ul{\Ical}(\bfz)$
independently of all other focus points $\bfz' \in \Omega$,
$\bfz' \neq \bfz$, are called \emph{beamforming methods}.
As they do not require evaluations of the source problem, a main
advantage of beamforming methods is that they are usually very fast.
On the other hand, these methods typically rely on heuristic
arguments and can only capture the main features of the source
power function rather than providing an exact reconstruction. 

Following the usual presentation in the field (see, e.g.,
\cite{Sijtsma2004}) a beamforming \emph{imaging functional} is defined
by
\begin{equation*}
  \ul{\Ical}_{\wvec}(\bfz)
  \,:=\, \wvec(\bfz)^* \Cmat \wvec(\bfz) \,,
  \qquad \bfz\in\Omega \,,
\end{equation*}
with a \emph{steering vector} $\wvec(\bfz) \in \C^M$ that depends on 
the \emph{focus point} $\bfz$ and is assumed to satisfy the constraint 
\begin{equation}
  \label{eq:UnitGain}
  \wvec(\bfz)^* \gvec(\bfz)
  \,=\, 1 \,. 
\end{equation}
The latter is often called \emph{unit gain}.
A particular beamforming method is therefore fully determined by its
steering vector. 
The steering vector of Capon's method is given by
\begin{equation}
  \label{eq:wcap}
  \wvec^{\mathrm{Cap}}(\bfz)
  \,:=\, \frac{\Cmat^\dagger \gvec(\bfz)}
  {\gvec(\bfz)^* \Cmat^\dagger \gvec(\bfz)} \,,
  \qquad \bfz\in\Omega \,.
\end{equation}
This yields the imaging functional $\Icap:\, \Omega \to \R$, 
\begin{equation}
  \label{eq:ICapon}
  \Icap(\bfz)
  \,:=\, \frac{\gvec(\bfz)^* \Cmat^\dagger}
  {\gvec(\bfz)^* \Cmat^\dagger \gvec(\bfz)}  \Cmat
  \frac{\Cmat^\dagger \gvec(\bfz)}{\gvec(\bfz)^* \Cmat^\dagger \gvec(\bfz)}
  \,=\, \left( \gvec(\bfz)^* \Cmat^\dagger \gvec(\bfz) \right)^{-1} \,,
  \qquad \bfz\in\Omega \,,
\end{equation}
which coincides with the discrete imaging functional $\Ifac$ of the
factorization method in \eqref{eq:imagingfunc_fac_discrete2}.

In the traditional derivation of Capon's method (see, e.g.,
\cite[p.~358]{Johnson1993}) it is assumed that the correlation
matrix $\Cmat$ is positive-definite, and the steering
vector~$\wvec^{\mathrm{Cap}}(\bfz)$ is obtained, for any
$\bfz\in\Omega$, as the solution of the constrained optimization
problem, 
\begin{equation}
  \label{eq:wcapon_def}
  \wvec^{\mathrm{Cap}}(\bfz)
  \,=\, \argmin \limits_{\wvec \in \C^M} \wvec^* \Cmat \wvec
  \quad \text{subject to} \quad
  \wvec^* \gvec(\bfz) = 1 \,.
\end{equation}
If $\gvec(\bfz)\in\ran(\Cmat)$, which is always the case when $\Cmat$
is positive-definite, then $\wvec^{\mathrm{Cap}}(\bfz)$
from~\eqref{eq:wcap} is a solution to \eqref{eq:wcapon_def} (see,
e.g., \cite[pp.~443--447]{NocWri99}). 
The minimization problem \eqref{eq:wcapon_def} is usually motivated as
follows. 
According to our model in Section~\ref{sec:aeroacoustic_ip} the
pressure signals at the microphone positions
$\bfx_1,\ldots,\bfx_M\in\M$ are zero-mean, complex random variables.
We collect them in a vector-valued random variable
$\pvec:=[p(\bfx_1),\ldots,p(\bfx_M)]^\top\in\C^M$ with
zero mean. 
Then $\Cov(\pvec) = \E(\pvec\pvec^*) = \Cmat$, and considering the
inner product with the steering vector $\wvec(\bfz)^*\pvec$, one seeks
to reduce noise as well as signals coming from other focus
points~${\bfz'\in\Omega}$,~${\bfz'\not=\bfz}$, whereas the signal
originating at the focus point $\bfz$ should not be dampened. 
The latter requirement is ensured by the unit gain
constraint~\eqref{eq:UnitGain}. 
The first requirement is enforced by minimizing the variance of
$\wvec(\bfz)^*\pvec$.
Therefore, Capon's method is also known as
\emph{minimum variance method}.
Minimizing the variance yields
\begin{equation*}
  \begin{split}
    \min_{\wvec\in\C^M} \var(\wvec^*\pvec)
    &\,=\, \min_{\wvec\in\C^M}
    \E\left(|\wvec^*\pvec|^2\right)
    \,=\, \min_{\wvec\in\C^M} \wvec^*
    \E\left(\pvec\pvec^*\right) \wvec 
    \,=\, \min_{\wvec\in\C^M} \wvec^* \Cmat \wvec \,,
  \end{split}
\end{equation*}
which explains the cost functional in \eqref{eq:wcapon_def}.

Finally, we note that Capon's beamformer can equivalently be written as 
\begin{equation*}
  \Icap(\bfz)
  \,=\, \inf \left \lbrace \wvec^* \Cmat \wvec \;:\;
    \wvec \in \C^M \,,\;
    \wvec^* \gvec(\bfz) = 1 \right\rbrace \,,
  \qquad \bfz\in\Omega \,,
\end{equation*}
whenever $\gvec(\bfz)\in\ran(\Cmat)$. 
This is the discrete analogue of the infimum in the inf-criterion of
the factorization method in Corollary~\ref{cor:InfCriterion}.

\section{Numerical examples}
\label{sec:numerical_examples}
We conclude our investigations with some numerical results for the
factorization method on experimental data, and we compare these
reconstructions to results that are obtained using two commonly used
conventional beamforming schemes. 
The dataset was measured at the cryogenic wind tunnel in Cologne
(DNW-KKK) on a $1:9.24$ scaled Dornier 728 half
model~\cite{Ahlefeldt2013}. 
Figure~\ref{fig:foto_experimental} shows the setup of this experiment.
The measurement array (on the right hand side of the picture) consists
of~$134$ microphones, which are flush-mounted at the wall of the wind
tunnel. 
The Mach number of the flow field is $|\bfm|=0.125$ (i.e., the wind
speed is $|\bfu|=43$~m/s), the angle of attack (i.e., the inclination
angle of the wing's cross section plane) is $9\degree$, and the
temperature is 11$\degree$C. 

The raw output data of the experiment consists of time series of
acoustic pressure fluctuations for each microphone with a total 
measurement time interval of $30$~s and a sampling frequency
of~$120$~kHz. 
These time series are then post-processed to obtain an estimated
correlation matrix~$\Cmat\in\C^{134\times134}$ using
Welch's method \cite{Welch1967} with a Hann weighting window, a
block size of $1024$ time samples and an overlap factor of 0.5.

\begin{figure}[H]
  \centering 
  \includegraphics[height = 7cm]{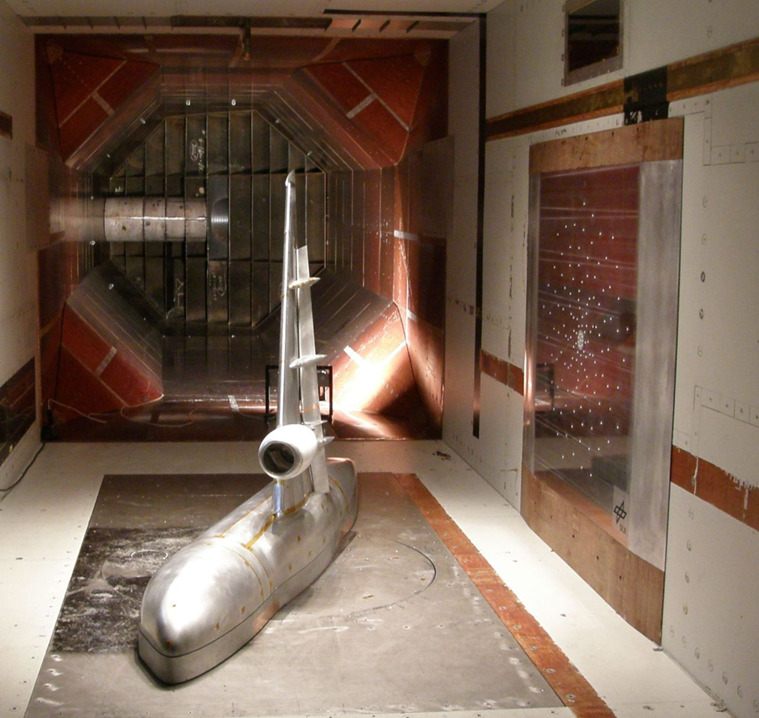}
  \caption{Photograph of the experimental setup with microphone array
    on the side wall~(courtesy of T.~Ahlefeldt, DLR G\"ottingen)}
  \label{fig:foto_experimental}
\end{figure}

We evaluate the imaging functional $\Ifac$ of the factorization method
from \eqref{eq:imagingfunc_fac_discrete1}, which coincides with
  the imaging functional of Capon's method from \eqref{eq:ICapon}, on
a two-dimensional plane~$\widetilde\Omega \subset \Omega$ that is
aligned to the cross-section of the aircraft wing.
The map size is $1.05$~m$\ \times \ 1.45$~m and the grid spacing is
$1$~cm.  
We compare these results to source maps obtained using two conventional
beamformers with and without diagonal removal (cf., e.g.,
\cite{Sijtsma2004}), which are defined by 
\begin{align}
  \Icbf(\bfz)
  &\,=\, \frac{\gvec(\bfz)^* \Cmat \gvec(\bfz)}{|\gvec(\bfz)|^4} \,,
  && \bfz \in \widetilde\Omega \,, \label{eq:ConvBeamformer}\\
  \Icbfdr(\bfz)
  &\,=\, \frac{\gvec(\bfz)^* \Cmat \gvec(\bfz)
    - \sum_{j=1}^M \Cmat_{jj} 
    |\gvec(\bfz)_j |^2 }{|\gvec(\bfz) |^4 - \sum_{j=1}^M |
    \gvec(\bfz)_j |^4 } \,,
  && \bfz \in \widetilde\Omega \,. \label{eq:ConvBeamformerDR}
\end{align}
  Diagonal removal is often used in experimental aeroacoustics to
  lower the effect of wind noise due to turbulent boundary layers
  directly at the microphone array (see, e.g., \cite{Sijtsma2004}).
  The imaging functionals in
  \eqref{eq:ConvBeamformer}--\eqref{eq:ConvBeamformerDR} can also be
  written as 
  \begin{align*}
  \Icbf(\bfz)
  &\,=\, \argmin_{\mu\in\R} \left\| \Cmat
    - \mu \gvec(\bfz)\gvec(\bfz)^* \right\|_F^2 \,,
  && \bfz \in \widetilde\Omega \,, \\
  \Icbfdr(\bfz)
    &\,=\, \argmin_{\mu\in\R} \sum_{\substack{j,\ell=1\\j\not=\ell}}^M
    \left| \Cmat_{j\ell} - \mu \gvec(\bfz)_j \ol{\gvec(\bfz)_\ell}
    \right|^2\,,
  && \bfz \in \widetilde\Omega \,. 
  \end{align*}
  where $\|\cdot\|_F$ denotes the Frobenius norm. 

To further reduce noise effects, the imaging outputs
$\ul{\Ical}(f,\bfz)$ of each imaging functional for single frequencies
$f$ are averaged over a frequency band $B$, i.e., we evaluate
the sum 
\begin{equation*}
  \ol{\ul{\Ical}}_B(\bfz)
  \,=\, \sum_{f \in B} \ul{\Ical}(f,\bfz) \,,
  \qquad \bfz \in \widetilde\Omega \,.
\end{equation*}
Here we consider third octave bands with center frequency
$f_{\nicefrac{1}{3}Oct}$, which are defined by 
\begin{equation*}
  B(f_{\nicefrac{1}{3}Oct})
  \,=\,  \left[2^{\nicefrac{-1}{6} }  f_{\nicefrac{1}{3}Oct} \,, 
    \ 2^{\nicefrac{1}{6} }  f_{\nicefrac{1}{3}Oct} \right] \,.
\end{equation*}
With a frequency resolution
$\Delta f = \frac{120\ \mathrm{kHz}}{1024} \approx 117$~Hz, the
number of discrete frequencies that are contained in a frequency band 
$B = [f_1, f_2]$ is given by 
\begin{equation}
  \label{eq:nof_in_band}
  \left \lfloor \frac{f_2}{\Delta f} \right \rfloor
  - \left \lceil \frac{f_1}{\Delta f} \right \rceil + 1 \,. 
\end{equation}
We note that in the notation of Section~\ref{sec:aeroacoustic_ip}
this corresponds to $\omega=2\pi f$, i.e., to a wave
number~${k = 2 \pi f/c}$, where $c\approx 345$~m/s is the speed of
sound. 
Moreover, the averaged imaging
values~$\ol{\ul{\Ical}}_{B(f_{\nicefrac{1}{3}Oct})}(\bfz)$ are 
normalized to the interval $[0,1]$ for each of the three methods. 
\begin{figure}
  \begin{subfigure}{.31\textwidth}
    \includegraphics[height = 1.05\textwidth]{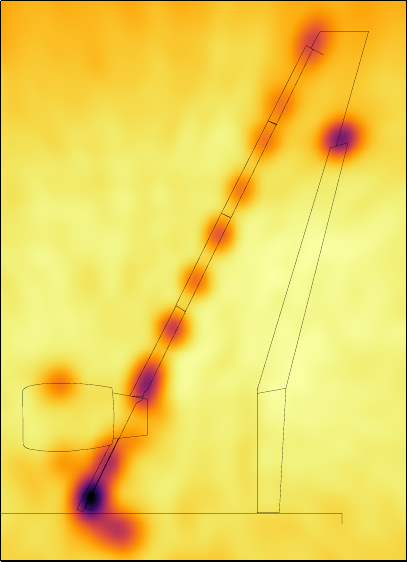}
    \caption{$f_{\nicefrac{1}{3}Oct} = 8\,\mathrm{kHz}$, \newline beamforming }
  \end{subfigure} 
  \begin{subfigure}{.31\textwidth}
    \includegraphics[height = 1.05\textwidth]{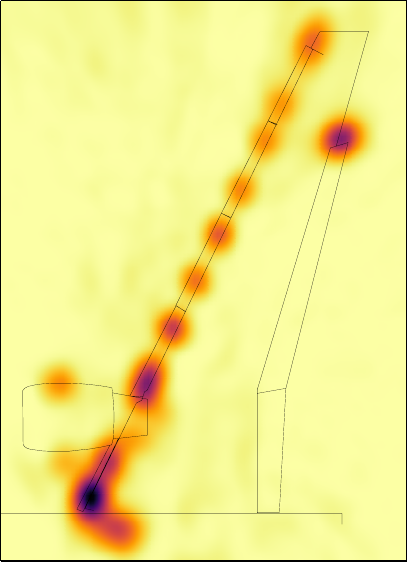}
    \caption{$f_{\nicefrac{1}{3}Oct} = 8\,\mathrm{kHz}$, \newline beamforming + DR}
  \end{subfigure}
  \begin{subfigure}{.31\textwidth}
    \includegraphics[height = 1.05\textwidth]{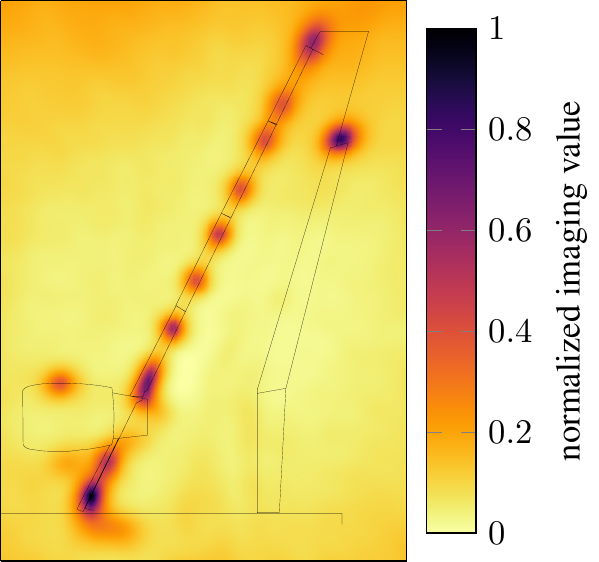}
    \caption{$f_{\nicefrac{1}{3}Oct} = 8\,\mathrm{kHz}$, \newline factorization method}
  \end{subfigure}
  \ \vspace{0.5cm}  \\ 
  \begin{subfigure}{.31\textwidth}
    \includegraphics[height = 1.05\textwidth]{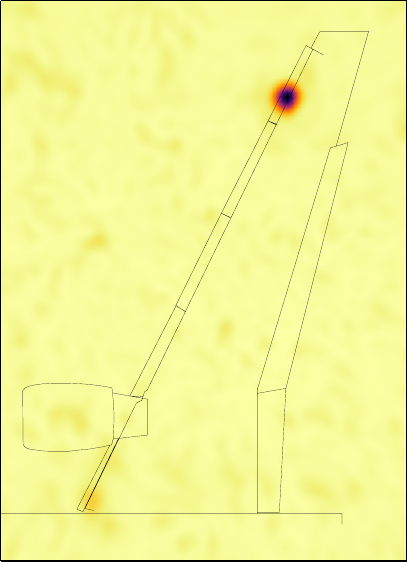}
    \caption{$f_{\nicefrac{1}{3}Oct} = 12\,\mathrm{kHz}$, \newline beamforming}
  \end{subfigure}
  \begin{subfigure}{.31\textwidth}
    \includegraphics[height = 1.05\textwidth]{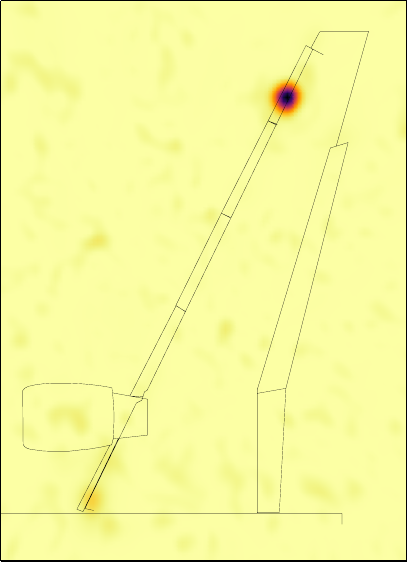}
    \caption{$f_{\nicefrac{1}{3}Oct} = 12\,\mathrm{kHz}$, \newline beamforming + DR}
  \end{subfigure}
    \begin{subfigure}{.31\textwidth}
    \includegraphics[height = 1.05\textwidth]{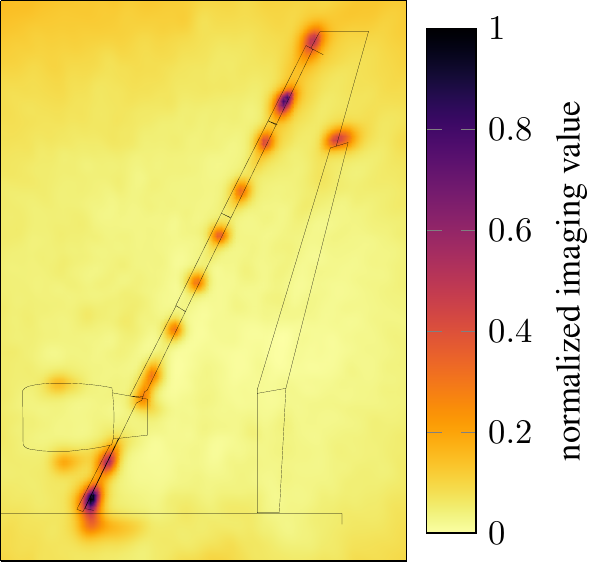}
    \caption{$f_{\nicefrac{1}{3}Oct} = 12\,\mathrm{kHz}$, \newline factorization method}
  \end{subfigure}
  \ \vspace{0.5cm}  \\ 
  \begin{subfigure}{.31\textwidth}
    \includegraphics[height = 1.05\textwidth]{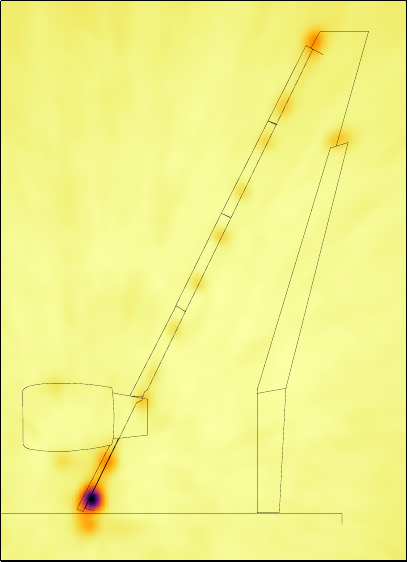}
    \caption{$f_{\nicefrac{1}{3}Oct} = 16\,\mathrm{kHz}$, \newline beamforming}
  \end{subfigure}
  \begin{subfigure}{.31\textwidth}
    \includegraphics[height = 1.05\textwidth]{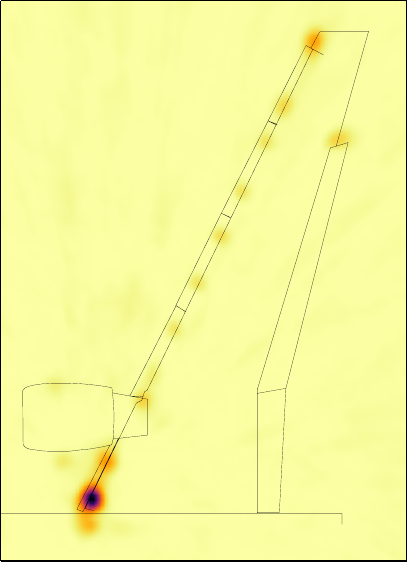}
    \caption{$f_{\nicefrac{1}{3}Oct} = 16\,\mathrm{kHz}$, \newline beamforming + DR}
  \end{subfigure}
  \begin{subfigure}{.31\textwidth}
    \includegraphics[height = 1.05\textwidth]{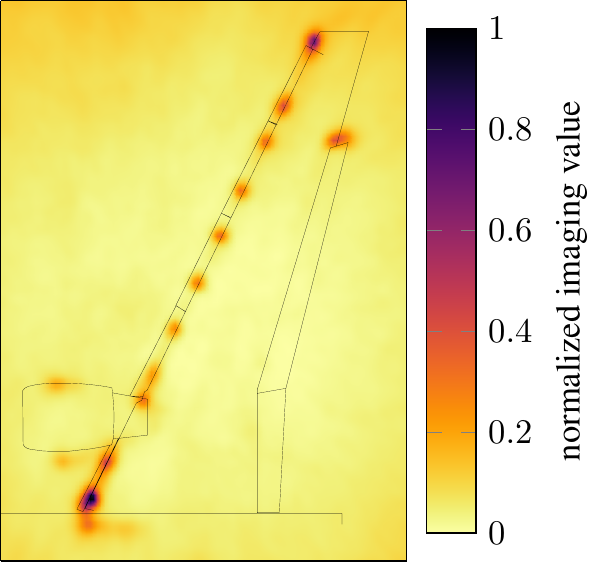}
    \caption{$f_{\nicefrac{1}{3}Oct} = 16\,\mathrm{kHz}$, \newline factorization method}
  \end{subfigure}
  \caption{
    Aeroacoustic source reconstructions for a Dornier-728 half-model,
    measured at the cryogenic wind tunnel in Cologne
    (DNW-KKK)~\cite{Ahlefeldt2013}. Source powers are shown on a
    cross-section through the wing and normalized to $[0,1]$.
    The Mach number is $0.125$, and the number of microphones $134$.}
  \label{fig:experimental_results}
\end{figure}

The results are shown for three third octave bands
($f_{\nicefrac{1}{3}Oct}=8$, $12$, and $16$~kHz) in Figure~\ref{fig:experimental_results}. 
According to \eqref{eq:nof_in_band} those frequency bands
contain $16$, $23$, and $32$ frequencies. 
At~$8$~kHz the factorization method provides a significant improvement
in spatial resolution when compared to conventional beamforming with
and without diagonal removal.
On the other hand, the reconstructions of the factorization method
contain more low frequent artifacts in regions apart from the wing,
where no sources are to be expected (e.g. in the top left corner of
the source maps).
At~$12$~kHz and~$16$~kHz only one or two dominating sources are
recovered by the conventional beamformers, while the factorization
method reconstructs regularly spaced sources on the leading edge of
the wing and a localized source at the end of the wing flaps.
All main source mechanisms are visible in the source maps of the
factorization method. 
The processing time for the factorization method is comparable to that
of the conventional beamformers.

  Recalling the equivalence of~\eqref{eq:imagingfunc_fac_discrete2}
  and \eqref{eq:ICapon} we conclude from \eqref{eq:wcapon_def} that
  the imaging functional of the factorization method (or equivalently
  of Capon's method) gives the smallest values among all beamformers
  maintaining the unit gain constraint~\eqref{eq:UnitGain}. 
  This cannot be seen in Figure~\ref{fig:experimental_results}
  directly, because all averaged imaging values have been normalized
  to the interval~$[0.1]$.
  However, it explains the higher resolution of the factorization
  method at $8$~kHz when compared to the conventional beamformers
  from~\eqref{eq:ConvBeamformer}--\eqref{eq:ConvBeamformerDR}.
  Using the orthonormal
  eigen\-system~$(\lambda_j, \ul{\psi}_j)_{1\leq j\leq M}$ of
  $\Cmat$ and denoting by $0< \npos \leq M$ the number of positive
  eigenvalues as before, the indicator functional $\Icbf(\bfz)$ from
  \eqref{eq:ConvBeamformer} can be written as
\begin{equation}
  \label{eq:imagingfunc_cbf_discrete} 
  \Icbf(\bfz)
  \,=\, \frac{1}{|\gvec(\bfz) |^4} 
  \sum_{j=1}^{\npos} \lambda_j
  \abs{\left\langle \gvec(\bfz), \ul{\psi}_j \right\rangle_{2}}^2 \,,
  \qquad \bfz\in\widetilde\Omega \,.
\end{equation}
Comparing this with \eqref{eq:imagingfunc_fac_discrete1} shows that
source components corresponding to large eigenvalues of the
correlation matrix dominate the reconstruction that is obtained by the
conventional beamformers, while the factorization method emphasizes on
source components related to smaller eigenvalues of the correlation
matrix. 
This, and our theoretical results from
Section~\ref{sec:factorization_aeroacoustics}, might be used to
explain the larger number of reconstructed source components that is
obtained by the factorization method in
Figure~\ref{fig:experimental_results} at $12$ and $16$~kHz. 
On the other hand, small eigenvalues of the correlation matrix do not
affect the stability of the conventional beamformers, while they lead
to instability of the factorization method (without further
regularization), which yields artifacts in reconstructions from
noisy data.

\section{Conclusions}
\label{sec:conclusions}
In this article, we have demonstrated that a variant of the
factorization method from inverse scattering theory can be used to
reconstruct the support of aeroacoustic random sources from
correlations of observed pressure fluctuations.
We established a rigorous characterization of the support of the
source power function in terms of the correlation data.

Moreover we have shown that the factorization method is closely
related to Capon's method, which is a well-established beamforming
method in experimental aeroacoustics. 
This unexpected relationship gives a new theoretically rigorous
interpretation of the reconstructions that are obtained by Capon's
method.
Our results basically say that Capon's method recovers the correct
support of the source power function, at least when the
latter is locally strictly positive.

\section*{Acknowledgments}
We are grateful to Thomas Ahlefeldt (DLR G\"ottingen) for sharing the
measurement data from~\cite{Ahlefeldt2013} and the picture of the
measurement setup in Figure~\ref{fig:foto_experimental} with us.
The first authors work was partially supported by the Deutsche
Forschungsgemeinschaft (DFG, German Research Foundation) -- Project-ID
258734477 -- SFB 1173. 
The second author would like to thank Thorsten Hohage
(Universit\"at G\"ottingen) for helpful discussions regarding the
theoretical analysis.
This research was initiated at the Oberwolfach Workshop
``Computational Inverse Problems for Partial Differential Equations''
in December 2020 organized by Liliana Borcea, Thorsten Hohage, and
Barbara Kaltenbacher.
We thank the organizers and the Oberwolfach Research Institute for
Mathematics (MFO) for the kind invitation. 
\FloatBarrier

{\small
  \bibliographystyle{abbrv}
  \bibliography{references}
}

\end{document}